\documentclass[11pt]{amsart}

\usepackage{amsmath,amssymb,latexsym}
\usepackage{amsfonts,amstext,amsthm,amscd}
\usepackage{graphicx}
\usepackage{mathrsfs}
\usepackage{hyperref}
\theoremstyle{plain}
\usepackage{color}
\usepackage[all]{xy}
\usepackage[top=1in, bottom=1in, left=1.2in, right=1.2in]{geometry}

\newtheorem{introtheorem}{Theorem}

\newtheorem{theorem}{Theorem}[section]
\newtheorem{corollary}[theorem]{Corollary}
\newtheorem{definition}[theorem]{Definition}

\newtheorem{lemma}[theorem]{Lemma}

\newtheorem{proposition}[theorem]{Proposition}

\newcommand{\bc}{\textbf{c}}
\newcommand{\ba}{\textbf{a}}

\newcommand{\parC}{\mathcal{C}}

\usepackage[draft]{changes}
\definechangesauthor[color=blue]{jk}
\definechangesauthor[color=red]{ae}

\def\C{\mathbb{C}}

\def\P{\mathbb{P}}

\def\N{\mathbb{N}}

\def\CP2{{\mathbb{CP}^2}}

\usepackage{calrsfs}
\DeclareMathAlphabet{\pazocal}{OMS}{zplm}{m}{n}

\def\cC{{\pazocal{C}}}
\def\cD{{\pazocal{D}}}

\def\cF{{\pazocal{F}}}
\def\cG{{\pazocal{G}}}

\def\cJ{{\pazocal{J}}}

\def\br{{\mathbf{r}}}
\def\bm{{\mathbf{m}}}
\def\bl{{\boldsymbol\ell}}
\def\balpha{{\boldsymbol\alpha}}

\def\boeta{{\boldsymbol\eta}}

\def\beps{{\boldsymbol\varepsilon}}
\def\bc{{\mathbf{c}}}

\def\aoneberk{\mathbb{A}^1_{an}}
\def\aone{\mathbb{A}^1}
\def\poneberk{\mathbb{P}^1_{an}}

\def\fatou{\cF}
\def\julia{\cJ}

\def\itin{\operatorname{itin}}

\def\fc{f_{\mathbf{c}}}

\begin{document}
\title{Oscillating wandering domains for $p$-adic transcendental entire maps}
\author{Adri\'an Esparza-Amador} 
\address{Instituto de Matem\'aticas, Pontificia Universidad Cat\'olica de
  Valpara\'iso}
\email{adrian.esparza@pucv.cl}
\author{Jan Kiwi}
\thanks{The second author was partially supported by CONICYT PIA ACT172001} 
 \address{Facultad de Matem\'aticas, Pontificia Universidad Cat\'olica de Chile}
\email{jkiwi@uc.cl}
\date{\today}
\maketitle
\begin{abstract}
  We give examples of transcendental entire maps over $\C_p$ having an oscillating wandering Fatou component.
\end{abstract}


\section{Introduction}
The general aim of this paper is to advance on the study of iterations of transcendental entire maps over the $p$-adic complex numbers $\C_p$.
The emphasis is on the Fatou-Julia theory, which was initially
considered by B\'{e}zivin in~\cite{Bezivin2001} and further developed by Fan and Wang in~\cite{FWTrans}.

As it is well established in one-dimensional non-Archimedean dynamics,
the action occurs on the corresponding Berkovich analytic space.
Consequently, a transcendental entire map $f$ will be regarded as a dynamical system acting on the Berkovich
affine line $\aoneberk$ over $\C_p$.
The \emph{Fatou set} of $f$, denoted $\fatou(f)$,
consists of the elements of $\aoneberk$ having a neighborhood
where the iterates of $f: \aoneberk \to \aoneberk$
form a normal family, in the sense of~\cite{FKTMontel}.
A connected component $U$ of $\fatou (f)$ is necessarily bounded (see~\cite[Theorem~6.5]{FKTMontel}), and 
$f(U)$ is again a component of $\fatou (f)$.
Such a component is called \emph{eventually periodic} if $f^{\circ n}(U) = f^{\circ m}(U)$
for some $n > m \ge 0$, and otherwise is called  \emph{wandering}.
According to Fan and Wang~\cite[Section 4]{FWTrans},
eventually periodic components must be Berkovich disks whose behavior is analogous to eventually periodic components of polynomial maps
(e.g~see~\cite{BenedettoBook}).
The focus of this article is on wandering Fatou components, sometimes simply called \emph{wandering domains}.

An a priori classification for wandering domains of entire maps, disregarding existence issues, is as follows. 
{We say that a} wandering Fatou component $U$ of an entire map $f$ is
 \emph{escaping} if $f^{\circ n} (x) \to \infty$, for all $x \in U$, and
 \emph{dynamically bounded}\footnote{When working over more general fields one has to further distinguish dynamically bounded wandering domains according to whether $(f^{\circ n} (x))$ converges to a type II periodic orbit or not.} if  $\{f^{\circ n} (x): n \in \N \}$
is precompact in $\aoneberk$, for all $x \in U$.
When there exists a point $x$ in $U$ whose orbit is unbounded but not
converging to $\infty$, then every point in $U$ has this property, and we say that $U$ is an \emph{oscillating wandering domain} (see~Proposition~\ref{p:bounded-escaping}).

Dynamically bounded examples are known or follow from the literature. In fact, 
Benedetto \cite{benedetto02,BenedettoWandering} provided the first examples of polynomials in $\C_p$ with a wandering Fatou component (necessarily a dynamically bounded disk).
According to Fern\'andez~\cite{Fernandez}, a large class of families obtained as perturbations of those considered by Benedetto  still possess members with wandering domains.
The existence of
transcendental entire maps
with dynamically bounded wandering domains follows.

{Examples of escaping wandering domains are also known and not hard to construct.}
According to Fan and Wang~\cite[Theorem~1.2]{FWTrans}, any Fatou component
which is not a Berkovich open disk is an escaping domain.
A transcendental entire map with escaping annuli Fatou components  was
provided in~\cite[Example~6.7]{FKTMontel}.  
Examples with escaping disk components are easy to produce. 
On the contrary, we do not know if every entire transcendental map has an escaping disk component.

The aim of this article is to provide the first examples of transcendental entire maps over $\C_p$ with oscillating wandering domains. We build on the {beautiful}
ideas introduced by Benedetto~\cite{benedetto02,BenedettoWandering} to produce (dynamically bounded) wandering domains in the context of polynomial dynamics.



\medskip
Given a strictly increasing sequence $\mathbf{r}= (R_j)_{j \geq1}$
of real numbers converging to infinity and bounded below by $p$,
we will consider a family of transcendental entire maps
parametrized by
$$\parC (\mathbf{r}) = \{ \mathbf{c} = (c_j)_{j \in \N} :  c_j \in \C_p
\text{ and } |c_j| =R_j \}.$$
Specifically, given $\bc = (c_j) \in \parC (\mathbf{r})$, let
\begin{equation}\label{eq:fc}
f_{\bc}(z):=\frac{z^p}{p}\prod_{j\geq1}\left(1-\frac{z}{c_j}\right).
\end{equation}

Under an extra mild condition on $\mathbf{r}=(R_j)$, we will say that $\br$ is a \emph{generic sequence of radii} (see Definition~\ref{d:generic}),
and our main result  proves that for any $\bc \in \parC(\mathbf{r})$ the map $f_{\bc}$ can be perturbed to one possessing
an oscillating wandering domain:

\begin{introtheorem}
  \label{ithr:A}
  Let  $\mathbf{r}$ be a generic sequence of radii and let $(\varepsilon_j)$ be a sequence of positive real numbers. 
  Then, for all $\bc = (c_j) \in \parC(\mathbf{r})$, 
  there exists $\bc'=(c_j') \in  \parC(\mathbf{r})$
  such that $f_{\bc'} $
  has an oscillating wandering domain
  and $|c_j - c_j'| < \varepsilon_j$, for all $j \ge 1$.
\end{introtheorem}

The local degree of $\fc$ at $z=0$ is $p$, for all $\bc$.
That is, all the maps $\fc$ possess a \emph{wild critical point} at $z=0$
whose dynamical influence plays a crucial role in the proof of our main result. 
Following Trucco~\cite{trucco}, we say that an entire transcendental map $f: \aoneberk \to \aoneberk$ is \emph{tame} if the local degree of $f$ at
any $y \in \aoneberk$ is not divisible by $p$. It is natural to ask if there
are tame transcendental entire maps with oscillating wandering Fatou components.

\medskip
Let us now briefly outline the organization of the paper.

In Section~\ref{s:preliminaries}, we review some fundamentals about the Berkovich affine line and analytic maps with the primary purpose of introducing notation. We refer the reader
to~\cite{BakerRumelyBook, BenedettoBook, MonographBerkovich, JonssonLNM} for a detailed exposition.
We also recall the definitions and main properties of the Julia and Fatou sets
following~\cite{FKTMontel} and~\cite{FWTrans}.

In Section~\ref{s:oscillating}, 
  we start discussing the main dynamical features of the maps $\fc$ and introduce symbolic dynamics according to a partition of $\aoneberk$ specially adapted for our purpose. Then, we give a detailed outline
  of the proof of Theorem~\ref{ithr:A}. More precisely, we discuss
  intermediate results and state,
   without proof,
   lemmas \ref{l:uniform}, \ref{l:stability}, and \ref{l:connecting}.
   Then, assuming these lemmas, we proceed to establish Theorem~\ref{thr:main}
   which is a quantified
   version of Theorem~\ref{ithr:A}.

   Sections~\ref{s:uniform}, \ref{s:stability}, and \ref{s:connecting} are devoted to proving  lemmas \ref{l:uniform}, \ref{l:stability}, and \ref{l:connecting}, respectively.


\subsection*{Acknowledgments}
The authors are grateful with Matthieu Arfeux for
motivating us to study non-Archimedean transcendental dynamics
in greater detail. 

\section{Preliminaries}
\label{s:preliminaries}

\subsection{Berkovich space}
For a detailed exposition on the Berkovich affine line, specially adapted to non-Archimedean dynamics, we refer the reader to~\cite[Part 3]{BenedettoBook}.
See also~\cite{FWTrans,FKTMontel,JonssonLNM}.
Our primary purpose here is to agree on notation.

For any $z \in \aone \equiv \C_p$ and $r >0$,
\begin{eqnarray*}
  D(z,r) &:= & \{ w \in \C_p : |w-z| < r \},\\
  \overline{D}(z,r) &:= & \{ w \in \C_p : |w-z| \le  r \},
\end{eqnarray*}
are \emph{open} and \emph{closed disks} in $\aone$, respectively.
The corresponding disks  in Berkovich space $\aoneberk$ are denoted by
$\cD(z,r)$ and $\overline{\cD}(z,r)$. 

Recall that the Berkovich affine line $\aoneberk$ is the space of multiplicative semi-norms in the ring $\C_p[z]$, that extend the $p$-adic absolute value in $\C_p$ endowed with the Gel'fand topology. With this topology $\aoneberk$ is a locally compact space.

The sup-norm  on $D(z,r)$, denoted $\xi_{z,r}$,  is often referred to as the point associated to $D(z,r)$.
Also, to each element $z$ of $\aone$ it corresponds the semi-norm given by
$f \mapsto |f(z)|$. Via this correspondence, the affine line $\aone$ is  identified with a (dense) subset of $\aoneberk$.

In general, for
$x \in \aoneberk$ and $f \in \C_p[z]$, it is convenient to denote
by $|f(x)|$ the semi-norm $x$ evaluated at $f$.
In particular, $|x|$ denotes the semi-norm $x$ evaluated at $f(z)=z$. 
Given $x \in \aoneberk$, let
$$\overline{\cD}_x := \{ y \in \aoneberk : |f(y)| \le |f(x)| \text{ for all } f
\in \C_p[z] \}. $$
According to Berkovich, $\overline{\cD}_x \cap \aone$ is either a singleton, a closed ball with radius  in $|\C_p^\times|$, a ball with radius not in $|\C_p^\times|$
or the empty set. The point $x$ is respectively called of
type I, II, III or IV. Type I points coincide with the elements of $\aone$
and sometimes are called ``classical points''.

Between any two points $x,y \in \aoneberk$, there is a unique arc connecting them, which we denote by $[x,y]$. It is convenient to consider the one point compactification $ \aoneberk \cup \{\infty\}$. 
We denote by $[x,\infty[$ the unique arc in $\aoneberk$ with one endpoint at $x$ that converges to $\infty$ at the other end.

Given $y \in \aoneberk$, two points $z, z'$, both distinct from $y$,
are said to be in the same \emph{direction at $y$}
if they lie in the same connected component of $\aoneberk \setminus \{y \}$.
Equivalently, $]z,y[\cap ]z',y[ \neq \emptyset$. 
The set of directions at $y$, denoted by $T_y \aoneberk$, is called the
\emph{(projectivized) tangent space or space of directions at $y$}.
Given a direction $\vec{v}$ at $y$, we denote by $\cD_y (\vec{v})$
the set of points in that direction.
There is a unique unbounded direction $\cD_y (\vec{v})$,
which we call the \emph{direction of $\infty$}. All other directions are
maximal open disks contained in
$\overline{\cD}_y$. In particular, the tangent space of
type I and IV points are trivial (i.e., a singleton). Type III points have two tangent directions. The tangent space of type II points is naturally endowed with
the structure of a projective line over the residue field $\widetilde{\C}_p$.

The action of an entire map $f: \C_p \to \C_p$ has a unique continuous extension to $\aoneberk$. This action preserves point types. Since the image of a disk in $\C_p$ is again a disk, the same occurs in $\aoneberk$. In fact, $f(\xi_{z,r}) = \xi_{w,s}$ if and only if $f(\overline{\cD}(z,r)) = \overline{\cD}(w,s)$.

An entire map $f: \aoneberk \to \aoneberk$ induces an action
$T_x f : T_x \aoneberk \to T_{f(x)} \aoneberk$ between tangent spaces for all $x \in \aoneberk$.
More precisely, given $\vec{v} \in  T_x \aoneberk$, the direction $\vec{w} =T_xf (\vec{v})$ is defined by the property that $f(\cD_x(\vec{v}) \cap U) \subset \cD_{f(x)}(\vec{w})$ for all sufficiently small neighborhoods $U$ of $x$. 
It follows that $T_xf$ maps a bounded direction (resp. the direction of infinity)
onto a bounded direction (resp. the direction of infinity).
Thus, this map is rather trivial for type I, III, and IV points.
For a type II point $x$, the tangent map $T_xf$ is a non-constant rational map between the corresponding tangent spaces endowed with their $\P^1 (\widetilde{\C}_p)$-structure.
Moreover, for any  bounded direction $\vec{v}$, we have
$f(\cD_x(\vec{v}))=\cD_{f(x)}(\vec{w})$ if and only if $\vec{w} =T_xf (\vec{v})$.

\subsection{Fatou-Julia sets}

Recall from~\cite{FKTMontel} that given an open subset $U$ of $\aoneberk$, we say that a family $\cG$ of analytic functions
$g: U \to \aoneberk$ is \emph{normal} if for any sequence $( g_n ) \subset \cG$
and any $x \in U$ there exists a neighborhood $V$ of $x$, and a subsequence
$(g_{n_j})$ such that $g_{n_j} : U \to \poneberk$ converges pointwise to a continuous function.
The main result in~\cite{FKTMontel} implies that
families $\cG$ for which there exists an open disk $\cD$ such that
$g(U) \cap \cD = \emptyset$, for all $g \in \cG$, are normal.

Given an entire function $f: \aoneberk \to \aoneberk$, we say that
a point $x \in \aoneberk$ lies in the \emph{Fatou set $\cF(f)$ of $f$} if
there exists a neighborhood of $x$ where $\{ f^{\circ n}: n \in \N \}$
is a normal family. The complement is the \emph{Julia set $\cJ(f)$ of $f$}.
Both of these sets are completely invariant under $f$.
and the Julia set is always non-empty. It is an open question posed by B\'{e}zivin~\cite{Bezivin2001} 
to determine if the Fatou set can be empty. See~\cite{FWTrans}
for a detailed discussion of
the Fatou and Julia sets of entire transcendental maps.
As in rational dynamics, periodic points of $f$ in $\aoneberk$ are classified
into attracting, indifferent, and repelling. According to Fan and Wang~\cite{FWTrans}, repelling periodic points are dense in the Julia set. 
Also,  $]x,\infty[ \cap \cJ(f) \neq \emptyset$ for all $x \in \aoneberk$ (\cite[Theorem~6.5]{FKTMontel}).

\medskip
As claimed in the introduction, the following result shows that once a point $z$ in a Fatou component $U$  has an unbounded non-escaping orbit, the orbit of all the points in $U$ have the same behavior:

\begin{proposition}\label{p:bounded-escaping}
  Let $U$ be a Fatou component of a transcendental entire map $f: \aoneberk \to
  \aoneberk$. Then, the following hold:
  \begin{enumerate}
  \item If there exists
  $z \in U$ with  bounded orbit, then there exists $R>0$ such that
  $f^{\circ n} (U) \subset \cD(0,R)$ for all $n \ge 0$.

\item
  If there exists  $z \in U$ such that $f^{\circ n}(z) \to \infty$, then
  for all $R>0$, there exists $N$ such that $f^{\circ n} (U) \subset \aoneberk \setminus \cD(0,R)$ for all $n \ge N$. 
  \end{enumerate}
\end{proposition}

\begin{proof}
(1)  If $z \in U$ has bounded orbit, then $f^{\circ n} (z) \in \cD(0,r)$ for some $r$ and all $n$. Hence, there exists $R > 0$ such that $\xi_{0,R} \in \cJ(f)$.
  Therefore, $f^{\circ n} (U) \subset {\cD}(0,R)$  for all $n$,  since $\cD(0,R)$  is a connected component of $\aoneberk \setminus \xi_{0,R}$.

 (2)  Let $ R > 0$ and
  assume that $f^{\circ n}(z) \to \infty$. There exists $R' >R$ such that $\xi_{0,R'} \in \cJ(f)$ and therefore, $f^{\circ n} (U) \subset \aoneberk \setminus \cD(0,R')$ for sufficiently large $n$.
\end{proof}

\subsection{Non-Archimedean Mean Value Theorem}
The usual Mean Value Theorem does not apply verbatim over $\C_p$.
For the sake of completeness we state  its non-Archimedean version below,
since it is intensively employed throughout this work.
The reader may find a detailed discussion in~\cite[Section~6.2.4]{robert}.

\begin{theorem}[Non-Archimedean Mean Value Theorem]
  \label{thr:rolle}
  Let $f: D(a,r) \to \C_p$ be an analytic map and set
  $\varrho := p^{-1/(p-1)}$. 
  Then, for all $z,w \in D(a,r)$ with $|z-w| < \varrho \, r$, there exists
  $u \in D(a,r)$ such that
  $$|f(z)-f(w)| = |f'(u)| \cdot |z-w|.$$
\end{theorem}

\section{Oscillating itineraries}
\label{s:oscillating}

This section is devoted to discussing the dynamics of the family of entire transcendental maps $\fc$ defined in the introduction.
We restrict to maps $\fc$ where the sequence $(R_j)$ is generic, in the sense
discussed in Section~\ref{s:generic}. Symbolic dynamics according to
a suitable  partition
of $\aoneberk$ is introduced in Section~\ref{s:symbolic}.
Specifically, in Section~\ref{s:cylinder}, we define the cylinder set associated to sequences $\bm$ and $\bl$ of positive integers. Oscillating wandering domains will arise as interior components of these cylinders after carefully choosing  $\bm$ and $\bl$, and proving that the interior is non-empty.
In Section~\ref{s:outline}, we give a detailed outline of the proof of Theorem~\ref{ithr:A} by reducing it to three main lemmas. 

\subsection{Generic sequence of radii}
\label{s:generic}
Let $\br = (R_j)$ be an (strictly) increasing sequence of real numbers converging to infinity, and for convenience,  bounded below by $p$ (i.e. $R_1 >p$).
Recall from the introduction that $\parC (\br)$ denotes the set of sequences $(c_j)$ such
that $|c_j|=R_j$, for all $j$, and for each $\bc \in \parC(\br)$,
we consider the map $\fc$ introduced in~\eqref{eq:fc}.

The arc $]0,\infty[ \subset \aoneberk$ is invariant under $\fc$ and
the action $\fc : ]0,\infty[ \to ]0,\infty[$ is described by
$$\varphi  (r) := \sup \{ |\fc(z)| : |z| \le r \},$$
where $r >0$. 
In fact, 
$$\fc (\xi_{0,r}) = \xi_{0,\varphi (r)}.$$
Often in the literature, 
 $\varphi$ is denoted by $|f_\bc|_r$.
The function $\varphi$ is increasing, piecewise monomial, and only depends on $\br$. Indeed, a rather straightforward calculation shows that
  for all $\bc \in \parC(\br)$,
  \[\varphi (r) =
\begin{cases}
  p \, r^p  & 0 < r \le R_1,\\
  {p } \, ({R_1 \cdots R_{j-1}})^{-1} \,  r^{p +j-1} & R_{j-1} < r \le R_j.
\end{cases}
\]

A sequence of radii will be called generic if the radii involved
are dynamically independent. More precisely:

\begin{definition}
  \label{d:generic}
  \emph{We say that  a sequence of real numbers $\br = (R_j)$ is a \emph{generic sequence of radii} if it is increasing, converges to $\infty$, $R_1 >p$, and for all integers $j > i \ge 1$,  we have:
  $$\varphi^{\circ n} (R_i) \neq R_j,$$
  for all $n \ge 1$.}
  \end{definition}

  Generic sequences exist and are, in a certain sense, ``generic'':
  
\begin{lemma}
  Given an increasing sequence of real numbers $\br = (R_j)$ such that $R_1>p$,
  and a sequence $(\delta_j)$ with $\delta_j >0$ for all $j$, there exists a generic sequence of radii $\br' = (R_j')$ such that $|R_j - R_j'| < \delta_j$ for all $j$.
\end{lemma}

\begin{proof}
  The $\varphi$-orbit of $R_i > \varrho$ is strictly increasing for
    all  $i \ge 1$. Moreover, $\varphi$ restricted to $[0,R_j[$
    is independent of $R_j$. Thus, given any $j \ge 2$, if necessary, we may find $R_j'$ arbitrarily close to $R_j$ so that $R_j'$ is not in the forward orbit of $R_i$ for all $i < j$. The lemma follows after recursively adjusting $R_j$ for all $j$.

\end{proof}

In the sequel, we let $\br = (R_j)$ be a generic sequence of radii.

\subsection{Symbolic Dynamics}
\label{s:symbolic}

\subsubsection{Dynamical partition}
Given $\bc\in\parC(\bc)$, to study the dynamics of $\fc$, we partition $\aoneberk$ into the sets:
\begin{eqnarray*}
  B_0 & := & \cD(0,1), \\
  B_j & := & \cD(c_j, R_j), \text{ for } j \ge 1 \text{, and}\\
  A & := & \aoneberk \setminus \left( \bigcup_{j \ge 0} B_j \right).
\end{eqnarray*}
We omit the dependence of $B_j$ and $A$ on $\bc$. In fact, unless otherwise stated, 
given a parameter $\bc$, we will only consider parameters $\bc'$ such that
$|c_j - c_j'| < R_j$. For all such $\bc'$ we have
$\cD(c_j, R_j)=\cD(c'_j, R_j)$.

Note that $\varphi:]0,\infty[ \to ]0,\infty[$ has a unique fixed point at 
$$\varrho:= p^{-1/(p-1)},$$
which is repelling.
It corresponds to
the type II point  $\xi_{0,\varrho} \in B_0$, fixed under the action of $\fc$, for all $\bc \in \cC(\br)$.
Each disk $B_j$, with $j\ge1$,  contains a type I repelling fixed point.
Julia fixed points of $\fc$ are in natural correspondence with
the disks $B_j$ for $j \ge 0$:


\begin{proposition}
  \label{p:fixed-points}
  For all $\bc \in \parC(\br)$ the following statements hold:
  \begin{enumerate}
  \item Every fixed point of $\fc$ is contained in $B_j$, for some $j \ge 0$.
  \item $B_0$ contains a unique repelling fixed point $w_0 = \xi_{0,\varrho}$.
    In appropriate coordinates, the tangent map $T_{w_0} \fc$ at $w_0$ is
    the Frobenius morphism.
  \item
    $B_j$ contains a unique fixed point $w_j (\bc)$, which is a repelling
    fixed point of type I with multiplier $\lambda_j$ such that
    $|\lambda_j| =\dfrac{\varphi(R_j)}{R_j}$.
  \end{enumerate}
\end{proposition}

We will consistently use the notation of the above proposition. That is, $w_j(\bc)$ will denote the fixed point of $\fc$ in $B_j$ with multiplier $\lambda_j$, for all $j \ge 1$.
We will ignore the dependence of  $\lambda_j$ on $\bc$,
since only its absolute value $|\lambda_j|$, which solely depends on $\br$,
will be relevant to us. The proposition will be a consequence of the following:

\begin{lemma}
  \label{l:absfc}
  For all $\bc \in \parC (\br)$,
  if $z \in A \cup B_0$,
  then  $$|\fc(z)|=  \varphi(|z|).$$
  Moreover, if $j \ge 1$, then $\fc: B_j \to \cD(0,\varphi(R_j))$ is an analytic isomorphism
  and, for all type I points $z,w \in B_j$,
$$|\fc(z) - \fc(w)| =  \dfrac{\varphi(R_j)}{R_j} \cdot |z -w|.$$
\end{lemma}

\begin{proof}
  We set $R_0:=0$ for convenience. Consider the auxiliary monomials $\mu_j$
  defined by $\mu_1(z) :=z^p/p$, and for $j \ge 2$, by
   $$\mu_j (z):= (-1)^{j-1} (p \cdot c_1 \cdots c_{j-1})^{-1}  z^{p+j-1}.$$
Let $z \in \aone$. For all $j \ge 1$, if 
  $R_{j-1} < |z| < R_j$, then
  \begin{equation*}
    |\fc (z) - \mu_j (z)| < |\fc(z)|
  \end{equation*}
  and, if $|z| = R_j$, then
  \begin{equation*}
    |\fc (z) -  \mu_j (z) \cdot c_j^{-1} (z-c_j)| < |\fc(z)|.
  \end{equation*}

Also, $|\mu_j(z)|=\varphi(|z|)$ if $R_{j-1} \le |z| \le R_j$. Therefore, $|\fc(z)|= \varphi(|z|)$ for all $z \in A \cup B_0$.
Moreover, on the disk $B_j \cap \aone = \{ z : |z-c_j| < R_j\}$,
the Weierstrass degree of $\fc$ is $1$, since it coincides with the
Weierstrass degree of $\mu_j (z) \cdot c_j^{-1} (z-c_j)$ on $B_j \cap \aone$.
Hence, $\fc: B_j \to \cD(0,\varphi(R_j))$
is an analytic isomorphism and the lemma follows. 
\end{proof}

\begin{proof}[Proof of Proposition~\ref{p:fixed-points}]
  For (1), just observe that
  $\varphi(|z|) > |z|$, for all $z \in A$.

  To prove (2), note that  if $z \in B_0$ and  $|z| >
  \varrho$, then  $|\fc(z)| = \varphi(|z|) > |z|$. Moreover, if $|z| \le \varrho$, then $ |\fc(z)| \le \varrho$. Hence, every point in $\overline{\cD}(0,\varrho)$
  distinct from $\xi_{0,\varrho}$ lies in the Fatou set.
  A computation shows that, after moving $\xi_{0,\varrho}$ to the Gauss point via $z \mapsto p^{-1/(p-1)} z$, the map $\fc$ reduces to $\tilde{z} \mapsto \tilde{z}^p$. Therefore,
  $\xi_{0,\varrho}$ is a repelling fixed point having as tangent map the Frobenius morphism. 

  Assertion (3) follows at once from the Lemma since the inverse of
  $\fc: B_j \cap \aone \to D(0,\varphi(R_j))$ is a contraction. 
\end{proof}

\subsubsection{Symbolic dynamics}
The dynamical partition of $\aoneberk$ furnishes a symbolic coding for the dynamics of $\fc$ via an itinerary function. Indeed, let
$$\Sigma := \{ A, B_0, B_1, \dots \}^{\N_0},$$
and
$$
\begin{array}[h]{rccl}
  \itin_\bc:& \aoneberk & \to &   \Sigma\\
 & x & \mapsto & (\alpha_n),\\
\end{array}
$$
if $\fc^{\circ n} (x) \in \alpha_n$.

We employ the usual multiplicative notation for concatenation
of symbols. For example, the fixed point $w_j(\bc)$ has itinerary
$B_j^\infty$. 

Given a finite or infinite word
$\boldsymbol{\alpha} = \alpha_0 \alpha_1 \dots $ in the alphabet $\{A, B_0, B_1, \dots\}$,
we denote by $|\balpha|$ its length (maybe $\infty$) and consider the associated
cylinder set of points having itinerary prescribed by $\balpha$:
$$C_\bc (\balpha) := \{ x \in \aoneberk : \fc^{\circ n} (x) \in \alpha_n \text{ for all } 0 \le n < |\balpha| +1 \},$$
with the understanding that $\infty +1 = \infty$.

There are some cylinder sets which are empty. 
In fact, given $j \ge 1$,
set
$$n_j := \min \{ n \ge 1 : f^{-(n+1)}_\bc (B_j) \cap B_0 \neq \emptyset \}.$$
That is, $n_j+1$ is the minimal number of iterations required for a point in
$B_0$ to reach $B_j$.
If the itinerary of
a point has a sub-word of the form $B_0 A^{n} B_j$ for some $n \ge 1$,
  then $n \ge n_j$.
Since $n_j$ is also the minimal $n$ such that $\varphi^{-(n+1)}_\bc(R_j) < 1$,
it is not difficult to see that  $n_j \to \infty$ as $j \to \infty$.
Recall that we require $R_1 >p$.
  This requirement is convenient since it forces $n_1 \ge 1$, and therefore
  $n_j \ge 1$ for all $j \ge 1$.

\subsubsection{Cylinder sets}
\label{s:cylinder}
  We adapt the remarkable strategy introduced by Benedetto
to produce (dynamically bounded) wandering domains in a
family of polynomial maps. We will obtain an
oscillating wandering domain  consisting of
  points with itinerary prescribed by two sequences
  $\bm = (m_j)_{j \ge 1}$ and $\bl= (\ell_j)_{j \ge 2}$ of positive
    integers. The  \emph{itinerary associated to $\bm$ and $\bl$} is 
$$\balpha = \balpha(\bm,\bl) : = B_0^{m_1} A^{n_1} B_1^{\ell_2} B_0^{m_2} A^{n_2} B_2^{\ell_3}
  B_0^{m_3} A^{n_3} B_3^{\ell_4}
  B_0^{m_4} A^{n_4} \cdots \in \Sigma.$$

  A point with itinerary $\balpha$ ``oscillates''. Indeed, the orbit starts at
  $B_0$ and then visits $B_1$ to come back to $B_0$ and then visit $B_2$, and so on. 
  
  \begin{proposition}
    \label{p:cylinder}
  Consider the itinerary $\balpha$ associated to  sequences $\bm = (m_j)_{j \ge 1}$ and $\bl=(\ell_j)_{j \ge 2}$ of positive
    integers. Let $\bc \in \parC (\br)$. 
Then $C_\bc (\balpha)$ is a non-empty closed set. Moreover, given a connected component $X$ of $C_\bc (\balpha)$ one of the following statements hold:
    \begin{itemize}
    \item $X = \{ y \}$ for some type I point $y$ in $\cJ(\fc)$.
\item $X$ is a closed disk whose boundary point $y$ is a type II or III point in $\cJ(\fc)$. Every connected component of the interior of $X$ is an oscillating  wandering domain.
    \item $X = \{y \}$ for some type IV point $y$ in $\cJ(\fc)$. 
    \end{itemize}
\end{proposition}

\begin{proof}
  For all $j \ge 1$, consider the word
  ${\boeta}_j=B_0^{m_j} A^{n_j} B_j^{\ell_{j+1}}$.
  Given $k \ge 1$, let $X_k = C_{\bc} (\boeta_1 \cdots \boeta_k B_0)$
  and observe that $\cap X_k = C_\bc (\balpha)$.

  We claim that $\cap X_k = \cap \overline{X_k}$.
  Indeed, if $x \in \partial X_k$ for some $k$, then  there exists $n \ge 0$ such that  $f^{\circ n}(x)$ is the
  boundary point of $B_{i}$ for some $i \le k$. Thus, $f^{\circ n+m} (x) \in A$ for all $m \ge 1$. Therefore, $ x \notin  C_\bc (\balpha) = \cap X_k  $ and the claim follows. 

  To show that $\overline{X_k}$ is non-empty we proceed recursively.
  Assume that $Y = C_\bc(\boeta_{j+1} \dots \boeta_k B_0)$ is not the empty set.
  Now we pull-back $Y$ according to $\boeta_j$
  in order to prove that the cylinder $C_\bc(\boeta_j \boeta_{j+1} \dots \boeta_k B_0)$ is also non-empty.
  More precisely, denote by $h_j$ the inverse of 
  $\fc: B_j \to \cD(0,\varphi(R_j))$. Let $Z = h_j^{\circ \ell_{j+1}} (Y)$.
  Consider
  $r_j = \varphi^{-(m_j+n_j)}(R_j) < \varrho$. At least one direction $\cD$
  at $\xi_{0,r_j}$ is such that it maps onto $B_j$ under  $\fc^{\circ m_j+n_j}$,
  that is,  $\fc^{\circ m_j+n_j} : \cD \to B_j$ is onto.
  Then the preimage  $Y' \subset \cD$  of $Z$ under this map 
  is such that $\emptyset \neq Y' \subset C_\bc (\boeta_j B_0)$ and $f^{\circ |\boeta_j|} (Y') = Y$.
  Hence,  $Y'$ is contained in 
  $C_\bc(\boeta_{j} \dots \boeta_k B_0)$.
  Thus, $\overline{X_k} \neq \emptyset$, for all $k \ge 1$.

  Let $X$ be a connected component of  $C_\bc(\balpha)$.
  Given $x \in X$, 
  if $y=\fc^{\circ n}(x)$ lies in $A$ or in some $B_j$, it follows that
  the associated disk $\overline{\cD}_y$ is contained in $A$ or in $B_j$, respectively, since $]0,\infty[$ is disjoint from  $C_\bc(\balpha)$.
  Therefore, if $x \in X$, then $\overline{\cD}_x \subset X$.
  Let $y$ be the point in $X$ with maximal diameter, which is unique,
  since $X$ is connected. Hence
  $\overline{\cD}_y = X$. It follows that $X$ is a type I or IV singleton, or a closed disk.

  To finish, we show that the singleton
  $\partial X$ lies in the Julia set. Suppose otherwise. Let $U$ be the
  Fatou component containing $\partial X$. If $U$ is a disk, then $\fc^{\circ n}(U) \cap [0,\infty[ = \emptyset$; therefore
  $U \subset C_\bc(\balpha)$. If $U$ is not a disk, then $U$ is an escaping Fatou component, according to~\cite{FWTrans}. Both alternatives lead to a contradiction so $\partial X \subset \julia (\fc)$.  
\end{proof}

\subsection{Outline of the Proof}
\label{s:outline}
The strategy to prove Theorem~\ref{ithr:A}
is to make appropriate choices of $\bm$ and $\bl$
so that, after perturbation of a given parameter $\bc$, a connected component of the cylinder set with itinerary $\balpha$ is neither a type I nor a type IV singleton. In this section, we introduce the definitions and notations required to  state three lemmas and explain how to deduce from them our main result. The rest of the paper is devoted to prove these lemmas.

\medskip
Consider two sequences
$\bm = (m_j)_{j \ge 1}$ and $\bl= (\ell_j)_{j \ge 2}$ of positive
integers, and let $\balpha$ be the associated itinerary.
Our first lemma will show that, under certain conditions, once
the cylinder set $C_\bc(\balpha)$ contains a type I point, it automatically contains a disk and therefore an oscillating wandering domain.
To explain this phenomenon, suppose that
$\bc$ is a parameter such that   $C_\bc(\balpha)$ contains a type I point $x$.
One should think of $m_j$ as the duration of the $j$-th excursion to the zone under the influence of the ramified type II fixed point $w_0$. The fixed point $w_0$ is
closely related to the wild critical point $z=0$. We think of these parts of the orbit as the \emph{wild excursions}. During the wild excursions the dynamics in $\aone$ is contracting.
The numbers $\ell_j$ should be thought as the duration of the  excursion near the $(j-1)$-th repelling fixed point. During these
excursions the dynamics in $\aone$ is expanding. Thus, these are the \emph{expanding excursions}. The numbers $n_j$ are transition times between the wild and expanding
excursions.
During the $j$-th wild excursion, we will show that the derivative along with
the orbit contracts by a factor of $p^{-m_j}$ (modulo constants). Along with the
expanding excursion that follows, around the fixed point $w_j(\bc)$, the derivative expands by a factor of $|\lambda_j|^{\ell_{j+1}}$ (modulo constants).
If  each wild excursion is long enough compared to  the expanding one that follows, then
the corresponding cylinder cannot be a type I singleton. Indeed, in Section~\ref{s:uniform}, we prove the following:

\begin{lemma}[Uniform disk]
  \label{l:uniform}
  Consider positive integer sequences 
  $\bl = (\ell_j), \bm = (m_j)$ with associated itinerary $\balpha$ 
  such that, for all $k \ge 1$,
  \begin{equation}
    \label{eq:sci}
    \varrho^{-1} R_k \, p^{-m_k} |\lambda_k|^{\ell_{k+1}}  <  1.
  \end{equation}
  If there exists a type I point $x$ and a parameter
  $\bc  \in \parC(\br)$ such that
  $$x \in  C_\bc (\balpha),$$
  then
  $$\cD(x,\varrho^2) \subset C_{\bc} (\balpha).$$
\end{lemma}

In view of Proposition~\ref{p:cylinder}, given $\bm$ and $\bl$ such that
~\eqref{eq:sci} holds for all $k$, to obtain a wandering domain, the
idea is to perturb an initial parameter $\bc$ to a nearby parameter $\bc'$ such that
$C_{\bc'} (\balpha)$ contains a type I point.

\subsubsection{Notation}
\label{s:notation}
Before we continue with the outline, let us introduce some notation.
Recall that $\br=(R_j)_{j \ge 1}$ is a generic sequence of radii as in
Definition~\ref{d:generic}. It will be convenient to set, for all $j \ge 1$,
$$S_j := \varphi^{-1} (R_j).$$

Consider sequences $\bm = (m_j)_{j\ge 1}$, $\bl=(\ell_j)_{j\ge 2}$
of positive integers with associated itinerary $\balpha$.
Given $k \ge 1$, the \emph{$k$-th truncation of $\balpha$} is the word:
$$\balpha^{(k)} = \balpha^{(k)}(\bm,\bl) := B_0^{m_1} A^{n_1} B_1^{\ell_2} B_0^{m_2} A^{n_2} B_2^{\ell_3}
\dots B_{k-1}^{\ell_k} B_0^{m_{k}} A^{n_{k}} B_k.$$
Let $L_1=0$, $M_1=m_1$, and for all $k \ge 1$, we set:
\begin{eqnarray*}
  N_k & :=& |\balpha^{(k)}| -1, \\
  L_{k+1} & :=& N_k + \ell_{k+1}, \\
  M_{k+1} & :=& L_{k+1} + m_{k+1}.
\end{eqnarray*}
Thus, $N_{k+1} = M_{k+1} + n_{k+1}$.
Given $j \ge k$,
note that the $L_k$-th iterate of an element in $C_\bc(\balpha^{(j)})$
is the first one of the $k$-th wild excursion. The $M_k$-th iterate is the beginning of the transition to the $k$-th expanding excursion that starts in the $N_k$-th iterate. 

Given  a parameter $\bc \in \parC(\br)$ and
a sequence $\beps = (\varepsilon_j)_{j \ge 1}$ of positive real numbers, we consider perturbations of $\bc$ in two type of sets.
That is, for $k \ge 2$, let
\begin{eqnarray*}
  U_k (\bc, \varepsilon_k)& := & \{ \bc' = (c_j') : |c_k' - c_k | < \varepsilon_k, c_j'= c_j \text{ if } j \neq k \},\\
\Delta_{k} (\bc, \beps) & := & \{ \bc' = (c'_j) : c_j = c_j'
  \text{ if } j < k \text{ and } |c_j - c_j'| < \varepsilon_j
  \text{ if } j \ge k \}.
\end{eqnarray*}

\medskip
In the sequel,  $\bm = (m_j)_{j\ge 1}$, $\bl=(\ell_j)_{j\ge 2}$ will be sequences of positive integers and $\beps= (\varepsilon_j)_{j \ge 1}$ will be a sequence of positive real numbers such that
$\varepsilon_j < R_j$ for all $j$. When clear from context, 
$\balpha$ will denote the itinerary associated to the sequences $\bm$ and $\bl$
with truncations $\balpha^{(k)}$. Also, $L_k, M_k$, and $N_k$  will denote the corresponding numbers introduced above.

\medskip
When the parameter $\bc \in \parC(\br)$ is clear from context and
  $y \in \aoneberk$, we freely employ $y_n$ to denote $\fc^{\circ n}(y)$.
  
  \subsubsection{Perturbation Lemmas}
Given an initial parameter $\bc$, the idea is to start with a type I point $x$ such that
$$\itin_\bc (x) = B_0^{m_1} A^{n_1} B_1^\infty.$$
That is, $x$ eventually maps onto the fixed point $w_1(\bc)$,
under iterations of $\fc$.
It is easy to show that such a point $x$ always exists. 
To prove our main result, we will successively perturb $\bc=\bc^{(1)}$
to obtain a sequence of parameters
$(\bc^{(k)})$ where  $\bc^{(k+1)}$ is obtained as a perturbation of $\bc^{(k)}$ for all $k\ge1$.
These parameters will be such that the
\emph{exact same point $x$} eventually maps, under $f_{\bc^{(k)}}$,
onto the fixed point $w_{k}(\bc^{(k)})$
in $B_k$ according to the itinerary
$$\itin_{\bc^{(k)}}  (x) = \balpha^{(k)} B_k^\infty.$$

\medskip
In Section~\ref{s:stability}, we give an explicit upper bound on the size of the perturbations to guarantee that the cylinder sets $C_\bc(\balpha^{(k)})$
remain stable. This will provide us a well defined range of perturbations
around $\bc^{(k)} $ to adjust the parameter without changing the initial segment of the itinerary of $x$. Recall that, for all $j \ge 1$, $S_j := \varphi^{-1} (R_j).$

 \begin{lemma}[Cylinder stability]
   \label{l:stability}
   Consider $k_0 \ge 1$. Let $\beps, \bl, \bm$ be such that the following hold:
      \begin{equation}
    \label{eq:stability}
\varepsilon_{k}  <   \varrho^{k-j} \cdot \dfrac{R_{k}^2}{R_{j-1} S_{j-1}} 
|\lambda_{j-1}|^{-\ell_{j}}  \text{ for all } 2 \le j < k,
\end{equation}
for all   $k > k_0 $
and, 
\begin{equation}
   \label{eq:transversality}
   \varrho^{-2} R_{j-1} S_{j-1} | \lambda_{j-1} |^{\ell_j} p^{-m_j}  <  1,
 \end{equation}
 for all $2 \le j \le k_0$.

If $\bc \in \parC(\br)$, then for all $\bc' \in \Delta_{k_0 +1}(\bc, \beps)$,
$$C_{\bc'} (\balpha^{(k_0)})= C_{\bc} (\balpha^{(k_0)}).$$
\end{lemma}

It will be convenient to simply say that ~\eqref{eq:transversality}
holds for some $k \ge 2$ if it holds for all $j \le k$.

The crucial observation here is that the upper bound in 
\eqref{eq:stability} on
$\varepsilon_k$
 \emph{does not depend on the lengths $m_j$ of the wild excursions}.

 \medskip
 For $\ell_{k+1}$ sufficiently large and any parameter $\bc \in \parC(\br)$, there exist points $y (\bc) \in B_{k}$ near to $w_k(\bc)$ with itinerary
 $B_k^{\ell_{k+1}} B_0^{m_{k+1}} A^{n_{k+1}} B_{k+1}^\infty$.
 To obtain $\bc^{(k+1)}$ from $\bc^{(k)}$, the idea is to ``connect'' the $N_k$-th iterate of $x$ with $y(\bc)$ for some $\bc$ close to $\bc^{(k)}$.
 More precisely, in Section~\ref{s:connecting}, we prove the following:
 
\begin{lemma}[Connecting]
  \label{l:connecting}
  Consider sequences
  $\beps, \bl, \bm$ and let $k \ge 1$. Assume that
  equations~\eqref{eq:stability} and \eqref{eq:transversality} hold for $k$
  and, 
     \begin{eqnarray}
      \label{eq:connecting}
      p \cdot \varrho^{-1} \dfrac{R_{k+1}^2}{S_{k}} |\lambda_{k}|^{-\ell_{k+1}}
      & < & \varepsilon_{k+1}.
    \end{eqnarray}
    If $\bc = (c_j) \in \parC(\br)$ is a parameter and
  $x$ is a type I point such that
  $$\itin_\bc(x)=  \balpha^{(k)} B_k^\infty,$$
  then there exists $\bc' = (c'_j) \in U_{k+1} (\bc, \varepsilon_{k+1})$  such that 
  $$\itin_{\bc'} (x) = \balpha^{(k+1)} B_{k+1}^\infty. $$
 \end{lemma}

 As noted above, the upper bound on the size of the perturbation
 required in ~\eqref{eq:stability} is independent of $\bm$.
 This  grants the possibility of constructing sequences that satisfy
 the hypothesis of lemmas~\ref{l:uniform}, \ref{l:connecting}, and \ref{l:stability}:

\begin{lemma}
  Let $(\bar{\varepsilon_i})$ be a sequence of positive real numbers.
  Then there exist $\beps = (\varepsilon_j)$ such that $0 < \varepsilon_j < \bar{\varepsilon}_j$ for all $j$ and, $\bm, \bl$ such that, for all $k \ge 1$,
  \eqref{eq:sci}, \eqref{eq:stability},  \eqref{eq:transversality}, and  \eqref{eq:connecting} hold.
\end{lemma}

\begin{proof} 
  We start by recursively choosing $\varepsilon_k$ and $\ell_k$ so that
  \eqref{eq:connecting} and \eqref{eq:stability} hold. That is,
  let $\varepsilon_2 = \bar{\varepsilon}_2$ and choose $\ell_2$ sufficiently large such that \eqref{eq:connecting} holds for $k=1$. 
  Now assume that, for some $k \ge 2$, we have already defined
  $\varepsilon_j$ and $\ell_j$, for all $j \le k$.
  Pick $0< \varepsilon_{k+1} < \bar{\varepsilon}_{k+1}$ so that \eqref{eq:stability} holds for $k+1$.
Now choose $\ell_{k+1}$ large enough so that \eqref{eq:connecting} holds. 

Once we have chosen $\beps = (\varepsilon_j)$ and $\bl = (\ell_j)$,  we finish by selecting, for all $j \ge 1$, integers $m_j$ sufficiently large so that  
both \eqref{eq:sci} and \eqref{eq:transversality} hold for all $k$. 
\end{proof}

Now we deduce from lemmas~\ref{l:uniform}, \ref{l:stability}, and
  \ref{l:connecting}
  a quantified version of our main result:

\begin{theorem}
  \label{thr:main}
  Consider 
  a generic sequence of radii  $\br$ and a parameter $\bc \in \parC(\br)$.
  Assume that $\beps, \bm, \bl$ are sequences such that   \eqref{eq:sci}, \eqref{eq:stability},  \eqref{eq:transversality}, and  \eqref{eq:connecting} hold for all $k \ge 1$.   Denote by $\balpha$ the itinerary associated to $\bm$ and $ \bl$. 

  If $x$ is a type I point such that
  $\itin_\bc (x) = B_0^{m_1} A^{n_1} B_1^\infty,$
  then there exists $\bc' \in \Delta_2(\bc,\beps)$ such that 
$$\itin_{\bc'} (x) = \balpha.$$
 Moreover, $x$ lies in an oscillating wandering Fatou component
of $f_{\bc'}$. 
\end{theorem}

\begin{proof}{(Assuming lemmas~\ref{l:uniform}, \ref{l:connecting}, and~\ref{l:stability})}
  Let $\bc^{(1)} = \bc$. For all $k \ge 1$, given a parameter
   $\bc^{(k)} \in \Delta_{2}(\bc,\beps)$ such that
  $x \in C_{\bc^{(k)}} (\balpha^{(k)})$, 
  apply Lemma~\ref{l:connecting} to
  obtain a parameter ${\bc^{(k+1)}} $ such that
  $x \in C_{\bc^{(k+1)}} (\balpha^{(k+1)})$.
  Let $\bc' = (c^{(j)}_j)$ and observe that
  $$\bc' \in \Delta_{k+1} (\bc^{(k)},\beps)$$ for all $k \ge 1$. Thus,  applying Lemma~\ref{l:stability}: 
  $$x  \in C_{\bc'} (\balpha^{(k)}),$$
  for all $k$. Therefore
  $$\itin_{\bc'} (x) = \balpha.$$ Hence, by Lemma~\ref{l:uniform}, 
  $$\cD(x,\varrho^2) \subset C_{\bc'} (\balpha).$$
  Finally, in view of Proposition~\ref{p:cylinder},
  the map $f_{\bc'}$ has an oscillating wandering domain that contains $x$. 
\end{proof}

\section{Uniform Disk}
\label{s:uniform}

The purpose of this section is to prove Lemma~\ref{l:uniform}.
The proof employs the following estimates on the derivative $\fc'$:

\begin{lemma}\label{l:fprima}
  Let $z \in \aone$. For all $\bc \in \parC(\br)$,
  $$|\fc' (z)| =
  \begin{cases}
    \dfrac{|f_{\bc}(z)|}{p|z|} & \text{ if } z \in B_0, \\
    |\lambda_j| & \text{ if } z \in B_j \text{ for some } j \ge 2,
  \end{cases}
  $$
  and if, $z \in A$, then
\begin{equation*}
|f_{\bc}'(z)| \le \dfrac{|f_{\bc}(z)|}{|z|}.
\end{equation*}
\end{lemma}

\begin{proof}
 By Lemma~\ref{l:absfc}, we only need to consider
   $z \in A \cup B_0$. For such $z$,
 the lemma is a consequence of the ultrametric inequality
  applied to the following formula: 
  \begin{equation*}\label{fprima}
f_{\bc}'(z)=f_{\bc}(z)\left[\sum_{j\geq1}\left(\frac{1}{z-c_j}\right)+\frac{p}{z}\right].
\end{equation*}
\end{proof}

\begin{lemma}
   \label{l:wc}
  Let $\bc \in \parC(\br)$ and $m, k \ge 1$. 
  If $z \in C_\bc(B_0^{m}A^{n_k}B_k)$ is a type I point, then
  \begin{equation*}
    D(z,\varrho) \subset C_\bc(B_0^m A^{n_k} B_k),
  \end{equation*}
   and 
  $$|(\fc^{\circ m +n_k})'(z)| < \varrho^{-1} R_{k} p^{-m}.$$
\end{lemma}


 

\begin{proof}
  Recall that  open disks map onto open disks under $\fc$.
  Moreover, given  $z \in \aone$, if $z \in A \cup B_0$, then
  $\fc (D(z,|z|))= D(\fc(z),|\fc(z)|) $.
  
  Let $z \in C_\bc(B_0^m A^{n_k} B_k)$ be a type I point.
  Then,  $D(z,|z|)$ maps, for all $n \le m+n_k$, onto the open
  disk $D(\fc^{\circ n} (z), |\fc^{\circ n} (z)|)$. 
  It follows that $D(z,\varrho) \subset C_\bc(B_0^m A^{n_k} B_k)$.
   Now, we apply Lemma~\ref{l:fprima} to  obtain
  $$|(\fc^{\circ m})'(z)| = \prod_{j=0}^{m-1}| \fc' (z_j)| =
  \prod_{j=0}^{m-1} \dfrac{|z_{j+1}|}{p |z_j|} = \dfrac{|z_{m}|}{p^{m} |z|} < \dfrac{|z_{m}|}{p^{m} \varrho} ,$$
   since 
   $z_j \in B_0$, for $0 \le j < m$, and $|z| > \varrho$.
  Taking into account that  $\fc^{\circ n}(z_{m}) \in A$ for $0 \le n \le n_k-1$,  Lemma~\ref{l:fprima} yields
  $$|(\fc^{\circ n_k})'(z_{m})| \le
  \dfrac{|z_{m+n_k}|}{|z_{m}|} = \dfrac{R_{k}}{|z_{m}|}.$$
  The inequality claimed in the statement of the lemma now follows from the chain rule.
\end{proof}

\begin{proof}[Proof of Lemma~\ref{l:uniform} (Uniform disks)]
  Assume that $x$ is a type I point in $ C_\bc (\balpha)$. 
  It is sufficient to prove, for all $k \ge 1$, the following inclusions: 
  \begin{eqnarray*}
    D(x, \varrho^2) &\subset &C_\bc(\balpha^{(k)}), \\
    \fc^{\circ L_k} (D(x, \varrho^2)) &\subset& D (x_{L_k}, \varrho^2).
  \end{eqnarray*}
  Indeed, if the first inclusion above holds for all $k \ge 1$, then the
  lemma follows.

  We proceed by induction. 
  The inclusions are easily verified for $k=1$ with the aid of Lemma~\ref{l:wc}
  and recalling that $L_1=0$.
  Suppose that the inclusions hold for $k$.
  Note that $D(x,\varrho^2) \subset C_\bc(\balpha^{(k+1)})$
  if and only if $$f^{\circ N_k} (D(x,\varrho^2)) \subset
  C_\bc( B_k^{\ell_{k+1}} B_0^{m_{k+1}} A^{n_{k+1}} B_{k+1}).$$
  Since $x \in C_\bc (\balpha)$,
  we have that $x_{L_k} \in C_\bc(B_0^{m_k} A^{n_k} B_k)$.
  By Lemma~\ref{l:wc}, $D(x_{L_k}, \varrho)$ is
  also contained in  this cylinder set and, for all $z \in D(x_{L_k}, \varrho)$, we have
  $$|(\fc^{\circ m_k+n_k})'(z)| < \varrho^{-1} R_k p^{-m_k} < |\lambda_k|^{-\ell_{k+1}},$$
  where the last inequality is granted by~\eqref{eq:sci}.
 From Theorem~\ref{thr:rolle},
  $$\fc^{\circ m_k+n_k} (D(x_{L_k}, \varrho^2)) \subset
  D(x_{N_k}, |\lambda_k|^{-\ell_{k+1}} \varrho^2).$$
  By Lemma~\ref{l:absfc}, it follows that 
  $$\fc^{\circ N_k + \ell} (D(x, \varrho^2)) \subset  \fc^{\circ \ell}
  (D (x_{N_{k}}, |\lambda_k|^{-\ell_{k+1}} \varrho^2)) =
  D (x_{N_k+\ell}, |\lambda_k|^{-\ell_{k+1}+\ell} \varrho^2),$$
  for all $0 \le \ell \le \ell_{k+1}$.
  Hence, $\fc^{\circ N_k}(D(x, \varrho^2)) \subset C_\bc(B_k^{\ell_{k+1}})$
  and
  $$ \fc^{L_{k+1}}(D (x, \varrho^2)) \subset
  D (x_{L_{k+1}}, \varrho^2) 
  \subset C_\bc(B_0^{m_{k+1}} A^{n_{k+1}} B_{k+1}),$$
  so $D(x,\varrho^2) \subset C_\bc(\balpha^{({k+1})}),$
  and the assertion holds for $k+1$.
\end{proof}

\section{Cylinder set stability}
\label{s:stability}
In this section, we prove Lemma~\ref{l:stability}.
The proof relies on first establishing stability results for certain cylinder sets  as well as studying the
dependence on parameters of points in an appropriate orbit.
This study  involves estimating the partial derivatives of $\bc \mapsto \fc^{\circ n}(z)$. The estimates  will be useful when combined with the non-Archimedean Mean Value Theorem (see~Theorem~\ref{thr:rolle}). 

We start by showing that
tangent maps, acting on a portion of the ``axis'' $]0,\infty[$,
remain constant under perturbations.
The stability of cylinders of the form $C_{\bc}(B_0^m A^{n_k} B_k)$ will then
follow.

\begin{lemma}
  \label{l:tangent}
  Let $k \ge 1$.
  Consider two parameters $\bc, \bc' \in \parC(\br)$ such that
  $|c_j - c_j' | < R_j$, for all $1 \le j < k$.
  Then, $$T_y \fc = T_y f_{\bc'}$$
  for all $y =\xi_{0,r} \in ]0,R_k[$. 
\end{lemma}
\begin{proof}
  For all $j \ge 1$, 
  let $$ \eta_j(z) = \dfrac{1-z/c_j'}{1-z/c_j} -1.$$

  Consider $z \in \aone$ such that $|z| < R_k$ and
  $z \in A \cup B_0$. A calculation shows that if $|z| \ge R_j$, then 
  $$|n_j(z)| = |1/c_j - 1/c_j'| R_j < 1,$$
  since $|c_j' - c_j | < R_j$. Moreover, if $|z| < R_j$, then 
  $$|n_j(z)| = |1/c_j - 1/c_j'|  < 1.$$
  Therefore,
  $$\left| 1 - \dfrac{f_{\bc'}(z)}{\fc (z)}  \right| = \left| 1- \prod_{j \ge 1} (1+\eta_j(z)) \right| < 1.$$

  Thus, given $0< r <R_k$, we have $|f_{\bc'}(z)-f_{\bc}(z)| < \varphi(r)$ for all
  $z$ in the set $ (A\cup B_0) \cap \overline{D}(0,r)$ which omits at most
  two directions at $y = \xi_{0,r}$. It follows that $T_y \fc = T_y f_{\bc'}$.
\end{proof}

\begin{corollary}
  \label{c:tangent}
  Consider $k \ge 1$.
  If $\bc, \bc' \in \parC(\br)$ are such that
  $|c_j - c_j' | < R_j$, for all $1 \le j < k$, then
  $$C_{\bc}(B_0^m A^{n_k} B_k) = C_{\bc'}(B_0^m A^{n_k} B_k),$$
  for all $m \ge 1$.
\end{corollary}
\begin{proof}
  Let $ r = \varphi^{-(m+n_k)}(R_k)$.
  Note that $z\in C_{\bc}(B_0^mA^{n_k}B_k)$
  if and only if $z$ lies in a direction
  at $\xi_{0,r}$  which maps under
  $\fc^{\circ m+n_k}$ onto $B_k$.
  By Lemma ~\ref{l:tangent}, these directions are independent of $\bc'$, under our assumptions on $\bc'$.
  
\end{proof}

The next result deals with the dependence of $\fc(z)$ on $c_j$ for points in $B_k$ when $j > k$.

\begin{lemma}
  \label{l:variationBk}
  Suppose that $\bc' \in U_{j}(\bc, R_j)$ for some $j >k$.
  Then, for all $z \in B_k \cap \aone$,
  $$|f_{\bc'} (z) - \fc (z) | =|\lambda_k|\cdot \dfrac{R_k^2}{R_j^2} |c_j -c_j'|.$$
\end{lemma}
\begin{proof}
  By a direct calculation, we have 
  $$ |f_{\bc'}(z)-f_\bc(z)|=\frac{\varphi(R_k)\cdot R_k}{R_j^2}|c_j-c_j'|. $$
  The proof concludes after replacing
   $\varphi(R_k)$ with $R_k \cdot |\lambda_k|$ (see~Lemma~\ref{p:fixed-points}). 
\end{proof}

In the sequel, we denote the partial derivative with respect to $c_j$ by $\partial_j$.
In particular, $\partial_j \fc^{\circ n} (z)$ denotes the partial derivatives of $\bc \mapsto \fc^{\circ n} (z)$, for certain $z \in \aone$ and $n \ge 1$, with respect to $c_j$ for some $j$.

Recall that we aim at showing that under sufficiently small perturbations,
the cylinder $C_\bc(\balpha)$ is constant.
Via partial derivatives, the following three results will
allow us to control the dependence on parameters of certain orbit elements.

\begin{proposition}
  \label{p:partialj}
  For all $z \in \aone$ such that $|z| < R_j$,
  $$|\partial_j \fc (z)| = \dfrac{|z| \cdot |\fc(z)|}{R_j^2}.$$
\end{proposition}

\begin{proof}
  By formula ~\eqref{eq:fc} of $\fc$, we have 
  $$
  \partial_jf_{\bc}(z)
	=\frac{zf_{\bc}(z)}{c_j^2(1-c_j^{-1}z)} 
	=\frac{zf_{\bc}(z)}{c_j(c_j-z)}. 
  $$
  Since $|z|<R_j$ implies that $|c_j-z|=|c_j|=R_j$, the proposition follows. 
\end{proof}

\begin{corollary}
  \label{c:partialj-fn}  
  Let $j \ge 1$ and $z\in \aone$ be such that $|z| > \varrho$.
  Assume that $\fc^{\circ n}(z) \in D(0,R_j) \setminus (B_1\cup \dots \cup B_{j-1})$ for all $0 \le n < N$.
  Then,
  $$|\partial_j \fc^{\circ N}(z)| = \dfrac{\varphi^{\circ N}(|z|) \cdot
    \varphi^{\circ N-1}(|z|)}{R_j^2}.$$
\end{corollary}

\begin{proof}
  We proceed by induction. Note that Proposition~\ref{p:partialj} establishes the case $N=1$. For the inductive step, first apply the chain rule:
  $$\partial_j\fc^{\circ 1+N}(z) = \fc'(\fc^{\circ N}(z)) \cdot
  \partial_j \fc^{\circ N}(z) +  \partial_j \fc ( \fc^{\circ N}(z)).$$
Then, observe that:
  \begin{eqnarray*}
     |\fc'(\fc^{\circ N}(z))| \cdot
     |\partial_j \fc^{\circ N}(z)| & \le &
     \dfrac{\varphi^{\circ N+1}(|z|)}{\varphi^{\circ N}(|z|)} \cdot \dfrac{\varphi^{\circ N}(|z|) \cdot \varphi^{\circ N-1}(|z|)}{R_j^2}\\
  &< &\dfrac{\varphi^{\circ N+1}(|z|) \cdot \varphi^{\circ N}(|z|)}{R_j^2}\\
                                   &= &|\partial_j \fc ( \fc^{\circ N}(z))|\\
                                   &= &|\partial_j\fc^{\circ N+1}(z)|,
  \end{eqnarray*}
  where the last line follows from the ultrametric triangle inequality.
\end{proof}

\begin{lemma}
\label{l:deltajNk}  Assume that $z$ is a type I point and 
  $$z \in C_\bc (\balpha^{(k)} B_k^\ell),$$
  for some $\ell \ge 0$.
  If \eqref{eq:transversality} holds for $k$, then 
  $$|\partial_j f^{\circ N_k + \ell'}_{\bc}(z)| = \dfrac{R_k S_k}{R_j^2} |\lambda_k|^{\ell'},$$
  for all $j > k$ and $0 \le \ell' \le \ell+1$.
\end{lemma}

\begin{proof}
  Given  a type I point $z \in C_\bc (\balpha^{(k)} B_k^\ell)$,
  let us first show that if the desired
  formula holds for $\ell'=0$, then it holds
  for $0 \le \ell' \le \ell+1$.
We proceed by induction on $\ell'$.
So assume that for some $0 \le \ell' < \ell+1$,
$$|\partial_j \fc^{\circ N_{k}+\ell'}(z)| = |\lambda_{k+1}|^{\ell'} \dfrac{R_{k}
  S_{k}}{R_j^2}.$$
Then, 
$$|\partial_j \fc^{\circ N_{k}+\ell'+1}(z)| = |\fc'(z_{N_{k}+\ell'})
\partial_j \fc^{\circ N_{k}+\ell'}(z) + \partial_j \fc (z_{N_{k}+\ell'})|,$$
and $$|\fc'(z_{N_{k}+\ell'})
\partial_j \fc^{\circ N_{k}+\ell'}(z)|= |\lambda_{k}| \cdot |\lambda_{k}|^{\ell'} \dfrac{R_{k} S_{k}}{R_j^2} >|\lambda_{k}|^{\ell'} \dfrac{R_{k}^2}{R_j^2}= |\partial_j \fc (z_{N_{k}+\ell'})|,$$
since $|\lambda_{k}|=\varphi(R_{k})/R_{k} > R_{k}/S_{k}$.
Therefore, $$|\partial_j \fc^{\circ N_{k}+\ell'+1}(z)| = |\lambda_{k}|^{\ell'+1} \dfrac{R_{k} S_{k}}{R_j^2}.$$
  
  Now we proceed by induction on $k$ to prove the lemma.
  If $z \in C_\bc (\balpha^{(1)} B_1^\ell)$, then, by Corollary~{\ref{c:partialj-fn}},
  $$|\partial_j \fc^{\circ N_1} (z)| = \dfrac{R_1 S_1}{R_j^2}.$$
  By the previous discussion we have the desired formula for $|\partial_j \fc^{\circ N_1+\ell'} (z)|$ when $0 \le \ell' \le \ell+1$.
  Suppose the lemma true for $k$ and that \eqref{eq:transversality} holds for $k+1$. For some $\ell \ge 0$, assume that
  $$z \in C_\bc (\balpha^{(k+1)} B_{k+1}^\ell).$$
  Then $$z \in C_\bc (\balpha^{(k)} B_k^{\ell_{k+1}-1}),$$
  so, by the inductive hypothesis,
  $$|\partial_j f^{\circ L_{k+1}}_{\bc}(z)| = \dfrac{R_k S_k}{R_j^2}
  |\lambda_k|^{\ell_{k+1}}.$$
 Recall that $M_{k+1} = m_{k+1} + L_{k+1}$ and apply the chain rule:
  \begin{eqnarray*}
    |\partial_j f^{\circ M_{k+1}}_{\bc}(z)| &=& |(\fc^{\circ m_{k+1}})' (z_{L_{k+1}}) \cdot \partial_j f^{\circ L_{k+1}}_{\bc} (z)
 +\partial_j f^{\circ m_{k+1}}_{\bc}(z_{L_{k+1}})|.
  \end{eqnarray*}
  From Lemma~\ref{l:fprima},
  $$|(\fc^{\circ m_{k+1}})' (z_{L_{k+1}})| \le p^{-m_{k+1}} \dfrac{|z_{M_{k+1}}|}{|z_{L_{k+1}}| }.$$
 Recall that $|z_{L_{k+1}}| > \varrho$. Therefore, the inductive hypothesis together with \eqref{eq:transversality} implies that:
$$|(\fc^{\circ m_{k+1}})' (z_{L_{k+1}}) \cdot \partial_j f^{\circ L_{k+1}}_{\bc} (z) | \le p^{-m_{k+1}} \dfrac{|z_{M_{k+1}}|}{|z_{L_{k+1}}| } \dfrac{R_k S_k}{R_j^2}
|\lambda_k|^{\ell_{k+1}}    < \dfrac{|z_{M_{k+1}}| \cdot \varrho}{R_j^2}.$$
 Corollary~{\ref{c:partialj-fn}}, applied to
  $z_{L_{k+1}} \in C_\bc (B_0^{m_k})$ gives:
$$|\partial_j f^{\circ m_{k+1}}_{\bc}(z_{L_{k+1}})| =
\dfrac{|z_{M_{k+1}}| \cdot|z_{M_{k+1}-1}|}{R_j^2}
> \dfrac{|z_{M_{k+1}}| \cdot \varrho}{R_j^2}. 
$$
Hence,
\begin{eqnarray*}
    |\partial_j f^{\circ M_{k+1}}_{\bc}(z)| &=&
 |\partial_j f^{\circ m_{k+1}}_{\bc}(z_{L_{k+1}})|.
\end{eqnarray*}

Now we apply the chain rule taking into account that $N_{k+1} = M_{k+1}+n_{k+1}$:
$$|\partial_j \fc^{\circ N_{k+1}}(z)| = |\partial_j  f^{\circ n_{k+1}}_{\bc}(z_{M_{k+1}}) +  (\fc^{\circ n_{k+1}})' (z_{M_{k+1}}) \cdot \partial_j f^{\circ M_{k+1}}_{\bc} (z) |$$
and, by Corollary~\ref{c:partialj-fn}, since
$z_{M_{k+1}} \in C_\bc(A^{n_{k+1}}B_{k+1})$, we have:
\begin{eqnarray*}
  |\partial_j  f^{\circ n_{k+1}}_{\bc}(z_{M_{k+1}})| &=&
  \dfrac{R_{k+1} S_{k+1}}{R_j^2}\\
  &>& \dfrac{R_{k+1}}{|z_{M_{k+1}}|} \cdot \dfrac{|z_{M_{k+1}}| \cdot|z_{M_{k+1}-1}|}{R_j^2}\\
  &\ge& | (\fc^{\circ n_{k+1}})' (z_{M_{k+1}}) \cdot \partial_j f^{\circ M_{k+1}}_{\bc} (z)|,
\end{eqnarray*}
where the last line is a consequence of Lemma~\ref{l:fprima}.
Thus,
$$|\partial_j \fc^{\circ N_{k+1}}(z)| = \dfrac{R_{k+1} S_{k+1}}{R_j^2}.$$

It follows that the lemma holds for $k+1$ and  $\ell'=0$. From the discussion
at the beginning of the proof, the lemma holds for $k+1$ and $0\le \ell'\le\ell+1$.
\end{proof}

We are now in position to apply the non-Archimedean Mean Value Theorem~\ref{thr:rolle} to prove the stability of the truncated cylinder
$C_\bc (\balpha^{(k)})$ under perturbations in
the set $U_{k+1}(\bc,\varepsilon_{k+1})$, introduced in Section~\ref{s:notation}.

\begin{lemma}
  \label{l:stability-with}
  Let $k\ge 1$. 
  Assume that \eqref{eq:stability} holds for $k+1$ and \eqref{eq:transversality}  holds for $k$.
  If  $\bc' \in U_{k+1}(\bc,\varepsilon_{k+1})$, 
  then $$C_{\bc'} (\balpha^{(k)}) = C_\bc (\balpha^{(k)}).$$
\end{lemma}

Before proving the lemma, let us observe that, by the Maximum Principle,  each connected component of the preimage of a Berkovich disk, under an entire map, is again a Berkovich disk.
In the case of our maps $\fc$, if $\cD$ is an open disk disjoint from $[0,\infty[$, then each component $\cD'$ of $\fc^{-1}(\cD)$, is again a disk disjoint from $[0,\infty[$. Thus, $\cD'$ is contained in $A$ or  $B_j$ for some $j \ge 0$. By repeatedly applying this observation, given $\bc$ and $k \ge 1$,
we conclude that each connected component
of $C_\bc(\balpha^{(k)})$ is an open disk. It follows that the cylinder
$C_\bc(\balpha^{(k)})$ is uniquely determined by
$C_\bc(\balpha^{(k)}) \cap \aone.$

\begin{proof}
  The lemma will follow once we have proven
  that, for all $1\le j \le k$, if $\bc' \in V_j:=U_{k+1}(\bc,\varrho^{j-k} \varepsilon_{k+1})$ and 
  $z \in C_\bc (\balpha^{(j)}) \cap \aone$, then $z \in C_{\bc'} (\balpha^{(j)})$.
  We will prove this assertion by induction on $j$.
  
  For $j=1$, the assertion holds by Corollary~\ref{c:tangent}.
  Suppose that it holds for some $j<k$. Let
  $z \in C_\bc (\balpha^{(j+1)})$ be a type I point and $\bc' \in V_{j+1}$.
  Then,  $z \in C_{\bc'} (\balpha^{(j)})$
  for all $\bc' \in  V_j$.
  From Lemma~\ref{l:deltajNk},
  $$|\partial_{k+1} f^{\circ N_j+1 }_{\bc'}(z)| =
  \dfrac{R_j S_j}{R_{k+1}^2} |\lambda_j|.$$
  In view of Theorem~\ref{thr:rolle} and \eqref{eq:stability}, for all $\bc' \in V_{j+1}$,
  \begin{eqnarray*}
    |f^{\circ N_j+1}_{\bc}(z)-f^{\circ N_j+1}_{\bc'}(z)|
          &\le& \varrho^{j-k+1} \varepsilon_{k+1} \cdot \dfrac{R_j S_j}{R_{k+1}^2} |\lambda_j|\\
          & < &\varrho |\lambda_{j}|^{-\ell_{j+1}+1} < R_j.
      \end{eqnarray*}
  In particular, $z \in C_{\bc'} (\balpha^{(j)}B_j)$.
  For $1 \le \ell \le \ell_{j+1}$,  we proceed recursively on $\ell$ to show that:
  \begin{equation}
    \label{eq:stability-proof}
    |f^{\circ N_j + \ell}_{\bc}(z)-f^{\circ N_j + \ell}_{\bc'}(z)|< \varrho
  |\lambda_{j}|^{-\ell_{j+1} + \ell}.
\end{equation}
In particular, this recursion will show that $z \in C_{\bc'} (\alpha^{(j)}B^{\ell-1}_j)$,
for all $1 \le \ell \le \ell_{j+1}$.
Assume that \eqref{eq:stability-proof} holds for some $\ell < \ell_{j+1}$.
For $z \in B_j$ and $\bc' \in V_{j+1}$, from Lemma~\ref{l:variationBk},
\begin{eqnarray*}
  |\fc(z) - f_{\bc'}(z)| &=&\dfrac{R_j^2}{R^2_{k+1}} |c_{k+1} - c_{k+1}'|\\
                         & < & \dfrac{R_j^2}{R^2_{k+1}}  \varrho^{-k+j+1} \varepsilon_{k+1}\\
                         &<& \dfrac{R_j^2}{R^2_{k+1}} \varrho \dfrac{R^2_{k+1}}{R_j S_j}  |\lambda_{j}|^{-\ell_{j+1}} < \varrho
                             |\lambda_{j}|^{-\ell_{j+1} + \ell+1},
\end{eqnarray*}
where the second line is obtained from the definition of $V_{j+1}$, and the third from the bound on $\varepsilon_{k+1}$ given by ~\eqref{eq:stability} and
then using that  $R_j/S_j < |\lambda_j|$.
Write $|f^{\circ N_j + \ell+1}_{\bc}(z)-f^{\circ N_j + \ell+1}_{\bc'}(z)|$
as $$|\fc(f^{\circ N_j + \ell}_{\bc}(z))-f_{\bc'}(f^{\circ N_j + \ell}_{\bc}(z)) +
f_{\bc'}(f^{\circ N_j + \ell}_{\bc}(z)) - f_{\bc'}(f^{\circ N_j + \ell}_{\bc'}(z))|.$$
Then,
\begin{eqnarray*}
  |f^{\circ N_j + \ell+1}_{\bc}(z)-f^{\circ N_j + \ell+1}_{\bc'}(z)| &<&
\max\left\{\dfrac{R_j^2}{R^2_{k+1}} |c_{k+1} - c_{k+1}'|, \varrho |\lambda_{j}|^{-\ell_{j+1} + \ell+1} \right\}\\
  &=& \varrho |\lambda_{j}|^{-\ell_{j+1} + \ell+1}.
\end{eqnarray*}
Thus, we establish  ~\eqref{eq:stability-proof} for $\ell =
\ell_{j+1}$. 
Hence, the distance between $f^{\circ N_j + \ell_{j+1}}_{\bc}(z)$ and $f^{\circ N_j + \ell_{j+1}}_{\bc'}(z)$ is bounded above by $\varrho$.
By Lemma~\ref{l:wc} and  Corollary~\ref{c:tangent},
$$f^{\circ N_j + \ell_{j+1}}_{\bc'}(z) \in C_\bc(B_0^{m_{j+1}} A^{n_{j+1}}B_{j+1})
=C_{\bc'}(B_0^{m_{j+1}} A^{n_{j+1}}B_{j+1}).$$
 It follows that
 $z \in C_{\bc'} (\balpha^{(j+1)})$, which ends the proof of the lemma. In fact,
 since $\bc \in U_{k+1}(\bc' , \varepsilon_{k+1})$, 
 the same argument proves that $C_{\bc'} (\balpha^{(k)}) \subset
 C_{\bc} (\balpha^{(k)})$.
\end{proof}

Our results so far are only concerned with perturbations in
$U_{k+1} (\bc,\varepsilon_{k+1})$, rather than in the larger set $\Delta_{k+1} (\bc, \beps)$, considered in the statement of
the Cylinder Stability Lemma~\ref{l:stability}.
In order to remedy this situation,
we show that maps in $\Delta_{k+1} (\bc, \beps)$ converge to $\fc$,
as $k \to \infty$, uniformly on any bounded set. 
To be precise, given $R > 0$ and an entire map $g$, let
$$ |g |_R:=\sup \{ |g(z)| : z \in D(0,R) \}.$$

\begin{lemma}
  \label{l:convergence}
  Assume that $\beps = (\varepsilon_{j})_{j \ge 2}$ converges to $0$, as
    $j \to \infty$. Then, given any $R > 0$, 
  $$\lim_{k \to \infty}
  \sup \left\{ \, |\fc - f_{\bc'} |_R : \bc' \in \Delta_{k}(\bc,\beps) \right\} = 0.$$
\end{lemma}

\begin{proof}
  Without loss of generality, we assume that $\varepsilon_j <1$ for all $j$,
  and $\bc' \in \Delta_2(\bc,\beps)$.
  Set $$a_j (z) := \dfrac{z}{c'_j} \dfrac{c_j - c_j'}{z - c_j},$$
  and observe that for all $|z| < R$, and all $j$ such that $R_j \ge R$,
  we have  $$|a_j (z)| < \dfrac{R}{R_j^2} \cdot \varepsilon_j <  \varepsilon_j.$$

  Now, if $\bc' \in \Delta_k(\bc,\beps)$ for some $k$, then
  $$\left|\dfrac{f_{\bc'}(z)}{f_{\bc} (z)} -1 \right| = \left|\prod_{j \ge k} (1 + a_j(z)) -1 \right| \le
  \max_{j \ge k} |a_j(z)| < \max_{j \ge k} \varepsilon_j,$$
  for all $|z| < R$.
Since $|f_{\bc}(z)| < \varphi(R)$ for all $|z| <R$, the lemma follows.
\end{proof}

\begin{proof}[Proof of Lemma~\ref{l:stability}]
  Consider $k_0 \ge 1$, and $\bc' \in \Delta_{k_0+1} (\bc,\beps)$.
   For $k \ge k_0$,
  let $\bc^{(k)}$ be the sequence defined by
  $c^{(k)}_j := c'_j$ for all $j < k$,
  and  $c^{(k)}_j := c_j$ for all $j \ge k$.
  Note that $\bc^{(k_0)} = \bc$.

  It follows that
  $\bc^{(k+1)} \in U_{k+1}(\bc^{(k)},\varepsilon_{k+1})$, for all $k \ge k_0$.
  Hence, for all $k > k_0$, by Lemma~\ref{l:stability-with},
  $C_{\bc^{(k)}} (\balpha^{(k_0)}) = C_{\bc}(\balpha^{(k_0)})$.
  Also, $\bc^{(k+1)} \in \Delta_{k+1} (\bc',\beps)$, so $f^{\circ n}_{\bc^{(k)}}(z) \to f^{\circ n}_{\bc'}(z)$, as $k \to \infty$,
  for all $z \in C_{\bc}(\balpha^{(k_0)})\cap \aone$, and all $n \le N_{k_0}$,
   by Lemma~\ref{l:convergence}. It follows that
   $C_{\bc'}(\balpha^{(k_0)}) \subset C_{\bc}(\balpha^{(k_0)}).$  Equality of the cylinder sets follows, since $\bc \in \Delta_{k_0+1} (\bc',\beps)$.
\end{proof}

\section{Connecting Lemma}
\label{s:connecting}
The Connecting Lemma~\ref{l:connecting} will be the outcome of two perturbations.
First, we consider a parameter $\bc$ and a point $x$ with itinerary $\balpha^{(k)}B_k^\infty$, that is, $\fc^{\circ N_k}(x)$ is the repelling fixed point in $B_k$.
Near $\fc^{\circ N_k}(x)$, there are points with itinerary $B_k^{\ell_{k+1}}B_0^\infty$. 
The idea is to obtain a parameter $\bc''$ perturbing $\bc$, to connect  $\fc^{\circ N_k}(x)$ with
a point with itinerary
$B_k^{\ell_{k+1}}B_0^\infty$. Then, the itinerary of $x$ under $f_{\bc''}$ becomes $\balpha^{(k)}B_k^{\ell_{k+1}}B_0^\infty$.
To achieve this, a priori estimates are provided in Lemma~\ref{l:h}  and
$\bc''$ is obtained in Lemma~\ref{l:cprime}.
Then, in the proof of the Connecting Lemma~\ref{l:connecting},
we produce a second perturbation that yields a parameter $\bc'$ near $\bc''$.
This second perturbation 
connects $f_{\bc''}^{\circ L_{k+1}} (x)$ with a 
point whose itinerary is $B^{m_{k+1}}_0 A^{n_{k+1}}B_{k+1}^\infty$.

\begin{lemma}
  \label{l:h}
  Given $k \ge 1$, let $$h_\bc (z) := \left( \fc|_{B_k} \right)^{-1}.$$
  Then, for all $\ell \ge 1$,
  $$|h_\bc^{\circ \ell} (0)- w_k(\bc)| = R_k |\lambda_k|^{-\ell},$$
  and 
  $$|\partial_{k+1} h_\bc^{\circ \ell} (0)| = \dfrac{R_k^2}{R_{k+1}^2 \cdot |\lambda_k|}.$$
\end{lemma}

\begin{proof}
  In view of Lemma~\ref{l:absfc}, the map $\fc$ expands distances by
  a factor of $|\lambda_k|$  in $B_k \cap \aone$. By induction, 
  $$|h_\bc^{\circ \ell} (0)- w_k(\bc)| = R_k |\lambda_k|^{-\ell},$$
  for all $\ell \ge 1$.

  From $\fc (h_\bc(0)) =0$, it follows that the lemma holds for $\ell =1$. Recursively, assume the lemma true for $\ell$.
  Since $h_\bc^{\circ \ell}(0)=f_\bc (h_\bc^{\circ \ell+1}(0))$,
  $$|\lambda_k| \cdot |\partial_{k+1} h_\bc^{\circ \ell+1} (0)| = |\partial_{k+1} h_\bc^{\circ \ell} (0) - (\partial_{k+1} \fc) (h_\bc^{\circ \ell+1}(0))|,$$
  and 
  $$|(\partial_{k+1} \fc) (h_\bc^{\circ \ell+1}(0))| = \dfrac{R_k^2}{R^2_{k+1}} >\dfrac{R_k^2}{R_{k+1}^2 \cdot |\lambda_k|} = |\partial_{k+1} h_\bc^{\circ \ell} (0)|.$$
 The desired formula for $|\partial_{k+1} h_\bc^{\circ \ell+1} (0)|$ follows.
\end{proof}

\begin{lemma}
  \label{l:cprime}
  Let $k \ge 1$, 
  Consider $\bc \in \parC (\br)$ and $x \in \aone$ such that
  $$f_{\bc}^{\circ N_k} (x) = w_k(\bc).$$
  Assume that \eqref{eq:stability} holds for $k+1$ and,  \eqref{eq:transversality}, \eqref{eq:connecting}  hold for $k$.
   Then, there exists $\bc'' \in U_{k+1}(\bc, \varepsilon_{k+1})$
  such that
  $$f_{\bc''}^{\circ N_k} (x) = h_{\bc''}^{\circ \ell_{k+1}} (0).$$
\end{lemma}

\begin{proof}
   Let  $H: U_{k+1}(\bc, \varepsilon_{k+1}) \to \aone$ be defined by
  $$H(\ba):= f_\ba^{\circ N_k}(x) - h_\ba^{\circ \ell_{k+1}} (0).$$
  For all $\ba \in U_{k+1}(\bc, \varepsilon_{k+1})$, combining the previous
  lemma with lemmas~\ref{l:deltajNk} and~\ref{l:stability-with},  we have that $x \in C_\ba (\balpha^{(k)})$ and
  $$|\partial_{k+1} f_\ba^{\circ N_k} (x)| = \dfrac{R_k S_k}{R_{k+1}^2} > \dfrac{R_k^2}{R_{k+1}^2 \cdot |\lambda_k|} = |\partial_{k+1} h_\ba^{\circ \ell_{k+1}} (0)|.$$
 Thus,
  $$|\partial_{k+1}H (\ba)| = \dfrac{R_k S_k}{R_{k+1}^2},$$
  for all $\ba \in  U_{k+1}(\bc, \varepsilon_{k+1})$.
  By  Theorem~\ref{thr:rolle},  
  $$H( U_{k+1}(\bc, \varrho \, \varepsilon_{k+1})) = D \left( H(\bc), \varrho \, \varepsilon_{k+1} \dfrac{R_k S_k}{R_{k+1}^2} \right).$$
  From Lemma~\ref{l:h} and \eqref{eq:connecting},
  $$|H(\bc)| = R_{k} |\lambda_{k+1}|^{-\ell_{k+1}} <  \varrho \, \varepsilon_{k+1} \dfrac{R_k S_k}{R_{k+1}^2}.$$
  Hence, there exists $\bc'' \in U_{k+1}(\bc, \varrho \, \varepsilon_{k+1})$ such that $H(\bc'') =0$ and the lemma follows.
\end{proof}

\begin{proof}[Proof of Lemma~\ref{l:connecting}]
  Let $r$ be such that
     $\varphi^{\circ\, n_{k+1}+m_{k+1}} (r) =R_{k+1}$. It follows that $\varrho < r < p$.
    Consider
    $$r' := r \dfrac{R_{k+1}^2}{R_k S_k} |\lambda_k|^{-\ell_{k+1}}.$$
    Note that, from ~\eqref{eq:connecting}, $r' < \varrho \varepsilon_{k+1}$.
    Consider $\bc'' = (c_j'')$ as in Lemma~\ref{l:cprime} and let
    $$\overline{U} := \{ \ba = (a_j) \in U_{k+1}(\bc'',\varepsilon_{k+1}) : |a_{k+1} - c''_{k+1} | \le r' \},$$
    $${U} := \{ \ba = (a_j) \in U_{k+1} (\bc'',\varepsilon_{k+1}) : |a_{k+1} - c''_{k+1} | < r' \}.$$
    Naturally, we may identify $\overline{U}$ and $U$ with $\overline{D}(0,r') \subset \C_p$
    and ${D}(0,r') \subset \C_p$, respectively.
    We will find a parameter $\bc' \in \overline{U} \setminus U$,
    as in the statement of the lemma.

    Consider
    $$
    \begin{array}[h]{rccl}
      F :& \overline{U}& \to & \C_p\\
      & \ba & \mapsto & f_\ba^{\circ L_{k+1}} (x).
    \end{array}
    $$
    
    \noindent
    \emph{Claim.} The map
    $F: \overline{U} \to \overline{D}(0,r)$ is an analytic isomorphism.

    \noindent
    \emph{Proof of the Claim.}
    Let
    $$\begin{array}[h]{rccl}
      F_0:& U_{k+1} (\bc', \varepsilon_{k+1})& \to & \C_p\\
      & \ba & \mapsto & f_\ba^{\circ N_{k}} (x).
    \end{array}
    $$
    By Lemma~\ref{l:deltajNk}, the derivative $F'_0(\ba)$ has constant
    absolute value $\dfrac{R_k S_k}{R_{k+1}^2}$. By Theorem~\ref{thr:rolle}, for all $\ba, \ba' \in \overline{U}$,
    $$|F_0 (\ba) - F_0(\ba')| = \dfrac{R_k S_k}{R_{k+1}^2} |\ba -\ba'|.$$
    From the definition of $r'$, it follows that $F_0$ maps $\overline{U}$ isomorphically onto a disk of radius $r |\lambda_k|^{-\ell_{k+1}}$, which is contained in $B_k$. 
    Now, for $1 \le \ell \le \ell_{k+1}$, consider $$
    \begin{array}[h]{rccl}
      F_\ell:& \overline{U}& \to & \C_p\\
      & \ba & \mapsto & f_\ba^{\circ N_{k} + \ell} (x).
    \end{array}
    $$
    The claim will follow once we prove the following assertion: for all  $0 \le \ell \le \ell_{k+1}$, $$|F_\ell (\ba) - F_\ell(\ba')| = |\lambda_k|^\ell \dfrac{R_k S_k}{R_{k+1}^2} |\ba -\ba'|,$$
    and $F_\ell$ is an  analytic isomorphism whose image is a ball of radius
    $r |\lambda_k|^{-\ell_{k+1}+\ell}$.
    Indeed, recursively, suppose  the assertion true for $\ell < \ell_{k+1}$.
    Then, $F_\ell (\bc'') = h_{\bc''}^{\circ \ell_{k+1} - \ell} (0) \in B_k$. Therefore,
    $F_\ell (\overline{U}) \subset B_k$ since 
    $r |\lambda_k|^{-\ell_{k+1}+\ell} < R_k$.
    Hence, for all $\ba, \ba' \in \overline{U}$,
    \begin{eqnarray*}
      |f_\ba (F_\ell(\ba)) - f_{\ba'}(F_\ell(\ba))| &=& \dfrac{R^2_k}{R_{k+1}^2}
                                                        |\ba-\ba'|\\
      & < & |\lambda_k|^{\ell+1} \dfrac{R_k S_k}{R_{k+1}^2} |\ba -\ba'|\\
      & = & |f_{\ba'} (F_\ell(\ba))- f_{\ba'} (F_\ell(\ba'))|,
    \end{eqnarray*}
    where the first line is from Lemma~\ref{l:variationBk}, the second line follows from $R_k/S_k < |\lambda_k|$, and the third line is a consequence of the assertion for $\ell$ and Lemma~\ref{l:absfc}.
    Then,
    \begin{eqnarray*}
      |F_{\ell+1} (\ba) - F_{\ell+1}(\ba')| &=& |f_\ba (F_\ell(\ba)) - f_{\ba'}(F_\ell(\ba)) + f_{\ba'} (F_\ell(\ba))- f_{\ba'} (F_\ell(\ba')) |\\
      &=&  |\lambda_k|^{\ell+1} \dfrac{R_k S_k}{R_{k+1}^2} |\ba -\ba'|,
    \end{eqnarray*}
and the claim follows. 

\medskip
Note that for all $\ba \in \overline{U} \setminus U$, we have
$|f^{\circ L_{k+1}}_\ba (x)| = |F_{\ell_{k+1}} (\ba)| = r$. Thus, the lemma reduces to proving the existence of a solution in $\overline{U} \setminus U$ to the equation:
$$f^{\circ N_{k+1}}_\ba (x) - w_{k+1}(\ba) =0.$$
    To show the existence of such a solution,
    given  $\ba \in \overline{U}$, let
    $$g_\ba:=f^{\circ\, n_{k+1}+m_{k+1}}_\ba : \overline{D}(0,r) \to  \overline{D}(0,R_{k+1}).$$
    Note that, for all $\ba \in \overline{U}$, the map $g_\ba$ is onto
    and its Weierstrass degrees  on $\overline{D}(0,r)$
    and ${D}(0,r)$ are both $\delta$, for some positive integer $\delta$ independent of $\ba$. Since $F: \overline{U} \to \overline{D}(0,r)$ has Weierstrass degrees $1$ both on $\overline{U}$ and on $U$, it is not difficult to check that the Weierstrass degrees of $\ba \mapsto g_\ba (F_\ba(x))$ on $\overline{U}$ and on $U$ are both $\delta$.

    To finish, observe that, for all $\ba \in   \overline{U}$, the fixed point $w_{k+1}(\ba)$ lies
    in $B_k$. In particular  $|w_{k+1}(\ba)|= R_k$.
    Therefore, the Weierstrass degree of
    $$\ba \to g_\ba (F(\ba)) - w_{k+1} (\ba) = f_\ba^{\circ N_{k+1}} (x) -w_{k+1} (\ba)$$
    on $\overline{U}$ is $\delta$ and on $U$ is $0$.
    Hence, there exists a parameter $\bc' \in \overline{U} \setminus U$
    which solves the equation and the lemma follows.
\end{proof}


\end{document}